\documentclass[12pt]{amsart}

\usepackage{amsmath,amssymb}
\usepackage{amsfonts}
\usepackage{amsthm}
\usepackage{latexsym}
\usepackage{graphicx}
\usepackage{epsfig}
\usepackage{enumitem}
\usepackage{xcolor}

\newtheorem{theorem}{Theorem}[section]
\newtheorem{lemma}[theorem]{Lemma}
\newtheorem{proposition}[theorem]{Proposition}
\newtheorem{corollary}[theorem]{Corollary}

\newtheorem{claim}{Claim}[section]
\newtheorem{definition}[claim]{Definition}

\newtheorem{prop}{Proposition}[section]
\newtheorem{remark}[prop]{Remark}

\makeatletter \@addtoreset{equation}{section} \makeatother

\def\RR{{\mathrm R}}
\def \2R{{\hat{\RR}}}

\def\Rc{{\mathrm {Rc}}}

\def\SS{{\mathrm S}}

\def\He{\mathrm {Hess}}

\def\id{\mathrm{Id}}

\def\lie{\mathcal{L}}
\def\tr{\text{tr}}

\usepackage[margin=1.19 in]{geometry}

\begin{document}

\title{K\"{a}hler Soliton Surfaces Are Generically Toric}



\author{Hung Tran}
\address{Department of Mathematics and Statistics,
	Texas Tech University, Lubbock, TX 79409}
\email{hung.tran@ttu.edu}
\thanks{$^*$Research partially supported by an NSF grant [DMS-2104988] and the Vietnam Institute for Advanced Study in Mathematics.}




%

\begin{abstract} Let $(M, g, \omega, f, \lambda)$ be a K\"{a}hler gradient Ricci soliton in real dimension four. One first observes that it is an integrable Hamiltonian system in a classical sense. Indeed, all known complete examples are toric and the symmetry is intrinsically related to the potential function $f$ and the scalar curvature $\SS$. While another article \cite{tran23contact} addresses the case that these functions are functionally dependent, this one considers the independent case. The main result states that the soliton admits a toric action under a generic assumption. That is, one assumes that the system is non-degenerate and the potential function $f$ is proper. Then there is an effective, completely integrable Hamiltonian toric $\mathbb{T}^2$- action on $(M, \omega)$. 
\end{abstract}
\maketitle
\section{Introduction}
The theory of Ricci flows, initiated by R. Hamilton in a series of articles including \cite{H3, HPCO, Hsurvey}, has received tremendous attention as a dynamical mechanism to explore the geometry and topology of a Riemannian manifold. Crucial to any of its applications is the understanding of singularity models, among which gradient Ricci solitons (GRS) play a fundamental role. A GRS $(M, g, f, \lambda)$ is a smooth manifold $M$ with Riemannian metric $g$, potential function $f$, and a constant $\lambda$ such that, for $\Rc$ denoting the Ricci curvature,  
\begin{equation}
	\label{grs}
	\Rc+\text{Hess}{f}=\lambda g. 
\end{equation}
By a combination of diffeomorphism and scaling, one can show that a GRS is a self-similar solution to the Ricci flow. Also, a GRS is a generalization of an Einstein metric. Furthermore, corresponding to the sign of $\lambda$, a GRS is called shrinking $(\lambda>0 )$, steady ($\lambda=0$ ), or expanding $(\lambda<0 )$. \\

In a complex setup, it is natural to consider a K\"{a}hler gradient Ricci soliton (KGRS) $(M, g, f, J, \lambda)$. That is, for an almost complex structure $J$, $(M, J, g)$ is K\"{a}hler and $(M, g, f, \lambda)$ is a GRS. KGRS arise naturally in the context of running a Ricci flow preserving a K\"{a}hler structure. Thus, this topic has extensive literature; see, for examples, \cite{tian1, WZ04toric, caohd09, CZ12Kahler, MW15topo, CS2018classification, CF16conical, CD2020expanding, DZ2020rigidity} and references therein. In particular, there are tremendous recent developments leading to the classification of all shrinking KGRS surfaces \cite{CDSexpandshriking19, CCD22finite, BCCD22KahlerRicci, LW23}. Their line of argument contains specifics which are only applicable to the case $\lambda>0$. 


In this paper, we propose an approach applicable for all signs of $\lambda$ and, thus, provide an important step towards a full classification of all complete KGRS in real dimension four. The key idea is based on the perspective of Hamiltonian dynamics and the detection of torus action. To describe the first result, recall that the K\"{a}hler form \[\omega:= g(\cdot, J\cdot)\]
of a  K\"{a}hler metric is closed. Thus, $(M, \omega)$ is a symplectic manifold. When coupled with a smooth function, it becomes a Hamiltonian system. Informally speaking, a system is integrable if it admits sufficiently many conserved quantities, called integrals of motion. This notion originated from the study of classical mechanics. See \cite{BF_book04} for a recent mathematical treatment of the subject.  \\

We first observe that a KGRS is an integrable system in the classical sense; additionally, the integrals of motion come from geometric functions. 

\begin{theorem}
	\label{main1}
	Let $(M, g, J, \omega, f, \lambda)$ be a KGRS in real dimension four. Then $f$ is a Morse-Bott function. Suppose that the potential function $f$ and the scalar curvature $\SS$ are functionally independent then they are integrals of motion for  $(M, \omega)$. Furthermore, if $f$ is proper then each level set is connected.
\end{theorem}

The alternative case-- $f$ and $\SS$ are functionally dependent-- is already treated in \cite{tran23contact}. Thus, we obtain the following immediate consequence. 

\begin{corollary}
	\label{cor1}
		 A complete KGRS $(M, g, J, \omega, f, \lambda)$ in real dimension four is an integrable Hamiltonian system. Furthermore, if $\nabla f$ is parallel to $\nabla \SS$ ($\nabla f\parallel \nabla \SS$) everywhere then the structure is either 
		 \begin{itemize} 
		 	\item a product of a constant curvature surface with a $2D$ K\"{a}hler GRS, or
		 	\item  of cohomogeneity one, $f$ is invariant by its action, and each principal orbit is a connected deformed homogeneous Sasakian structure. 
		 \end{itemize}
		  Otherwise, the integrals of motion are given by $f$ and $\SS$. If $f$ is proper then the system is complete.  
\end{corollary}

\begin{remark}
	Generally speaking, a Morse-Bott function is a function such that its Hessian is non-degenerate on the normal bundle to each connected component of a singular level set. Morse theory and its generalization on Morse-Bott functions, initiated by R. Bott \cite{bott54}, are fundamental in the study of integrable Hamiltonian systems. 
\end{remark}
\begin{remark}
	For a KGRS, the fact that $f$ is Morse-Bott is well-known to experts \cite{CST09note,CDSexpandshriking19}. The connectedness of its level set is an immediate consequence due to \cite{PRV15} and might be of independent interest.
\end{remark}

\begin{remark}
	See Section \ref{preliminary} for precise definitions of these terminologies. 
\end{remark}

For the rest of this article, our focus is on the case that $f$ and $\SS$ are functionally independent. The function 
\[\Phi: M \mapsto \mathbb{R}^2,~~ x\mapsto (f(x),\SS(x)) \]
is called the moment map. The general guideline is that significant data about the topology and geometry of the manifold is encoded in the image of $\Phi$.  For example, any compact component of the pre-image of a regular value, $\Phi^{-1}(c)$, is a Lagrangian $2$-diemsnional torus (Langrangian means the  restriction of $\omega$ to the tangent space is vanishing). Furthermore, for a completely integrable system with a proper moment map, the well-known Liouville-Arnold theorem guarantees a local torus action. That is, at a regular point, there is a neighborhood such that one can construct a torus $\mathbb{T}^2$-action preserving the system.\\

The existence of a global torus action is delicate and highly non-trivial. This is equivalent to extending the action across singular points and J. Duistermaat \cite{Duis80} first observed there is some monodromy. The monodromy is described in terms of a covering group such that its non-triviality is a topological obstruction to go from local to global. To understand such behavior, it is inevitable to study the set of singular points where the differential of $\Phi$ is not of the maximal rank. \\

Consequently, it is considered generic to assume that the moment map is proper and each singular point is non-degenerate. For a differentiable function, a generic singularity is a Morse one in the sense that the Hessian of the function is non-degenerate at that point. In considering an integrable Hamiltonian system with a tuple of functions, it is possible to extend that generic notion. Roughly speaking, for a system with two degrees of freedom, non-degeneracy implies that the restriction of one function on a regular level set of another is a Morse-Bott function. A precise definition in terms of Cartan sub-algebras is given in Section \ref{preliminary}. In case of a KGRS, the assumption can be equally formulated int terms of the Riemannian curvature tensor at each singular point. \\

Our second theorem asserts that there is a global torus action under such a generic assumption.

\begin{theorem}
	\label{main2}
	Let $(M, g, J, f, \lambda)$ be a complete KGRS in real dimension four. Supposed that the integrable Hamiltonian system $(M, \omega, f, \SS)$ is non-degenerate and $f$ is proper. Then $f$ is Morse and the moment map is of toric type. That is, there is an effective, completely integrable Hamitonian toric $\mathbb{T}^2$-action on M whose momentum map is of the form $\Upsilon=\ell \circ \Phi$, where $\ell$ is a diffeomorphism from $\Phi(M)$ into its image.  
\end{theorem}
\begin{remark}
	Indeed, all examples of complete KGRS in real dimension four are toric and the corresponding Halmitonian systems are non-degenerate. Our result is certainly consistent with the classification of \cite{CDSexpandshriking19, CCD22finite, BCCD22KahlerRicci, LW23}.  
\end{remark}
\begin{remark}
	A Morse function is a Morse-Bott function such that each connected component of a singular set is an isolated point. 
\end{remark}

\begin{remark}
	It is natural to ask if the torus action preserves the Riemannian metric $g$. It will be addressed elsewhere. 
\end{remark}

\begin{remark}
	Based on Theorem \ref{main2}, it is natural to expect that there is a $\mathbb{T}^2$-action preserving both the symplectic form $\omega$ and the metric $g$. Indeed, for a compact K\"{a}hler-Einstein maniford with non-zero scalar curvature, Y. Matsushima showed that the Lie algebra of Killing vector fields is a real form of the Lie algebra of real holomorphic vector fields \cite{Mat57}. Very recently, an analogous theorem is obtained by R. Conlon-A. Deruelle-S. Sun for a shrinking KGRS with bounded Ricci curvature \cite{CDSexpandshriking19}. Consequently, C. Cifarelli confirmed that a holomorphic torus action leads to an isometric one for that particular case \cite{cifarelli22}. The more general setup is non-trivial and will be addressed elsewhere.  
\end{remark}

The organization of the paper is as follows. Section \ref{preliminary} will recall preliminary results with an emphasis on describing notions related to Hamiltonian dynamics. The proofs of Theorem \ref{main1} and Corollary \ref{cor1} are given in Section \ref{integrable} by relatively simple algebraic observations. The rest of the paper is devoted to prove Theorem \ref{main2}. Section \ref{aboutHessS} derives a general formula for the Hessian of the scalar curvature in relation to one of the potential function. Then, in Section \ref{sectiongeo}, one observes that singular points of the system form geodesic curves and totally geodesic surfaces. The proof of Theorem \ref{main2} is given in Section \ref{singularity} based on the following steps:
\begin{itemize}
	\item General strategy: The singularity points are classified into various types (depending on combinations of blocs of either elliptic, hyperbolic, or focus-focus type) and a toric system generally corresponds to only elliptic singularities. 
	\item Step 1: Show that the rank 0 singularities are elliptic via direct calculation of $\He f$ and $\He \SS$ at these points (Theorem \ref{S0}). In particular, there are no focus-focus, hyperbolic-hyperbolic, or elliptic-hyperbolic singularities. 
	\item Step 2: Show that if a rank 0 singularity point $m$ is connected to a branch of rank 1 singularities then it determines an eigenspace of $\He f$ at $m$. As $\He f$ commutes with $J$, $\He f$ has at most two distinct eigenvalues and, thus, $m$ can connect to at most two branches of singularities (Corollary \ref{atmost2}).
	\item Step 3: Show that the rank 1 singularities are elliptic via a contradiction argument and a global consideration (Theorem \ref{S1elliptic}). That is, if a singular point has a hyperbolic bloc, then there is an embedded line segment of critical values in the interior of the image of the moment map. One extends this segment in both directions and the compactness deduces that it must stop at a rank 0 singularity $m$. By Step 1, $m$ is elliptic and connects to two branches of elliptic singularities. In addition to the hyperbolic branch, $m$ is connected to three branches, a contradiction to Step 2.   
	\item Step 4: Since all singularities are elliptic, we apply the theory on almost-toric integrable system developed by \cite{VNS07, PRV15, PV09}. 
\end{itemize}
 
\subsection{Acknowledgment}  The author benefits greatly from discussion with Profs. Quo-Shin Chi, Xiaodong Cao, Reyer Sjamaar, and Rui Loja Fernandes. We would like to thank an anonymous referee for constructive comments.
.  
\section{Preliminaries}
\label{preliminary}

In this section, we fix our notation and convention which will be used throughout the article and recall preliminary results. A major emphasis will be on describing notions related to Hamiltonian dynamics.  

\subsection{Notation and Convention}
Let $(M, g)$ be an orientable connected Riemannian manifold of an even dimension. Let $\nabla$ denote the unique Levi-Civita connection induced by $g$. For vector fields $X$ and $Y$, recall that $\nabla_{X} Y$ is the directional derivative of $Y$ along integral curves of $X$. An almost complex structure $J$ is defined to be a smooth section of the bundle of endormorphisms $\text{End}(TM)$ such that
\[J^2=-\id. \]
$J$ is said to be integrable if it is indeed induced from an atlas of complex charts with holomorphic transition functions. $(M, g, J)$ is called an almost Hermitian manifold and $g$ a Hermitian metric if
\[g(JX, JY)=g(X, Y).\] 
The K\"{a}hler form is defined via the almost complex structure:
\begin{align*}
	\omega(X, Y) &:= g(X, JY)
\end{align*}
A triple $(M, g, J)$ is called almost K\"{a}hler if $d\omega=0$. Furthermore, in case $J$ is integrable, it is called K\"{a}hler. From a Riemannian geometry perspective, the following is well-known.
\begin{proposition}\cite[Proposition 3.1.9]{BGbookSasakian08}
	\label{charKahler}
	Let $(M, g, J)$ be an almost Hermitian real manifold. The followings are equivalent:
	\begin{enumerate} [label=(\roman*)]
		\item $\nabla J=0$,
		\item $\nabla \omega_g=0$,
		\item $(M, g, J)$ is K\"{a}hler.  
	\end{enumerate}
\end{proposition} 

Next we recall various notions of curvature. First, the Lie bracket of vector fields is defined by, for any smooth function $h$,
\[(\lie_X Y) h= [X, Y] h= XY h- YX h.\]
Then, the Riemannian curvature, as an $(1, 1)$ tensor, is given by
\begin{align*}
	\RR(X, Y):&=[\nabla_X, \nabla_Y]-\nabla_{[X, Y]}.
\end{align*}
Our convention of the $(4,0)$ Riemannian curvature tensor is in agreement with \cite{KNvolumeI96, pe06book}: 
\[\RR(X, Y, W, Z):= g(\RR(X, Y)W, Z).\]
Consequently, the sectional curvature of the plane spanned by vectors $X$ and $Y$ is obtained by:
\[\text{sect}(X, Y):=\frac{\RR(X, Y, Y, X)}{g(X, X)g(Y, Y)-g(X, Y)^2.}\]
Then, the Ricci curvature is defined as a trace:
\[\Rc(X, Y):= \text{tr}(Z\mapsto \RR(Z, X)Y).\]

Moreover, for a $(k, 0)$ tensor, we recall the formula for the covariant derivative:
\begin{align*}
	(\nabla T)(X, Y_1,... Y_k) &:= (\nabla_X T)(Y_1,... Y_k),\\
	(\nabla_X T)( Y_1,...Y_k) &:=\nabla_X (T(Y_1,...Y_k))-\sum_{i=1}^k T(Y_1,...\nabla_X Y_i,...Y_k).
\end{align*}
Then $\delta$ denotes the divergence operator or the co-differential. That is, for an orthonormal basis,
\[ \delta T (\cdot)= \sum_{i} g((\nabla_{e_i} T)(\cdot), e_i).\]
$\imath$ denotes the interior product
\[\imath_X T(\cdot):= T(X, \cdot).\]

Let $h$ be a smooth function $h: M\mapsto \mathbb{R}$, its gradient is also the dual, via the Riemannian metric $g$, of the $1$-form $df$. That is, for any vector field $X$,
\[\nabla_X h= g(\nabla h, X)= dh(X).\]
The Hessian is then given by
\begin{align*}
	 \He h(X, Y) &:= g(\nabla_X\nabla h, Y)=g(\nabla_Y \nabla h, X)=\frac{1}{2}(\lie_{\nabla h} g) (X, Y),\\
	 &=X (Y f)-df(\nabla_X Y).
	 \end{align*}
The Laplacian is then just the trace:
\[\Delta h:= \tr (\He f).\]

Finally, we recall, For any vector field $X$, the derivation $A_X= \lie_X-\nabla_X$ is induced by a tensor field of type $(1, 1)$. If $X$ is a Killing vector field, by \cite{KNvolumeI96}, 
\begin{align*}
	\nabla_Y (A_X)&= \RR(X, Y).
\end{align*}

\subsection{ K\"{a}hler Gradient Ricci Solitons}
A Riemannian manifold $(M^n, g)$ with a function $f: M\mapsto \mathbb{R}$ is called a gradient Ricci soliton (GRS) if
\begin{equation*}
	\text{Rc}+\frac{1}{2}\lie_{\nabla f} g =\Rc+ \He f= \lambda g.
\end{equation*}
Then, the first and second Bianchi's identities of a Riemannian curvature tensor lead to several consequences: 
\begin{align}
	\label{deltaS}
	\SS + \triangle f  &= n\lambda;\\
	\label{rcandf}
	\Rc(\nabla{f}) &=\frac{1}{2}\nabla{\SS}=\delta \Rc;\\
	\label{nablafandS}
	\SS+|\nabla f|^2-2\lambda f &= \text{constant};\\
	\label{lapS}
	\triangle\SS+2|\text{Rc}|^2 &= \left\langle{\nabla f,\nabla \SS}\right\rangle+2\lambda \SS.
\end{align}
For a proof, see \cite{chowluni}. Also, (\ref{nablafandS}) is generally considered a conservation law. We will then collect a few facts that will be used later:

\begin{itemize}
	\item If the soliton is geodesically complete, then the vector field $\nabla f$ is complete \cite{zhang09completeness}.
	\item If $\lambda \geq 0$, then $\SS\geq 0$ by the maximum principle and equation $(\ref{lapS})$.  Moreover, such a complete GRS has positive scalar curvature unless it is isometric to the flat Euclidean space  \cite{zhang09completeness, chenbl09}.
	\item If $\lambda<0$, then $\SS$ is also bounded below \cite{zhang09completeness}. As a consequence, equation (\ref{nablafandS}) implies that $f$ is bounded above. 
	\item The Riemannian metric $g$ and the potential function $f$ of a gradient Ricci soliton are real analytic \cite{bando87, Ilocal}. See also \cite[Lemma 3.2]{DW11coho}. 
	
\end{itemize}





In the presence of a complex structure, it is natural to define a K\"{a}hler gradient Ricci soliton (KGRS).
\begin{definition}
	$(M, g, J, f)$ is a K\"{a}hler GRS if $(M, g, f)$ is a GRS and $(M, g, J)$ is a K\"{a}hler manifold.
\end{definition}

It is crucial to observe that, on a K\"{a}hler manifold $(M, g, J)$, $\Rc$ is $J$-invariant. Thus, for a K\"{a}hler GRS, so is $\He f$. The following is well-known.  

\begin{lemma}
	\label{killing1}
	Let $(M, g, J)$ be a K\"{a}hler manifold and $f:M \mapsto \mathbb{R}$ such that $\He f$ is $J$-invariant. Then, we have the followings:
	\begin{enumerate} [label=(\roman*)]
		\item $J(\nabla f)$ is a Killing vector field.
		\item $\nabla f$ is an infinitesimal automorphism of $J$. 
	\end{enumerate}
\end{lemma}

\subsection{Morse and Morse-Bott Functions}
The theory of Morse and Morse-Bott functions will be fundamental to our approach. Our references are \cite{Nico07book, PRV15}.  Throughout this section, $h$ is a smooth function $h: M\mapsto \mathbb{R}$. 

\begin{definition}
	A critical point of a smooth function $h: M\mapsto \mathbb{R}$ is non-degenerate if its Hessian at that point is non-degenerate. A smooth function is called Morse if all its critical points are non-degenerate.  
\end{definition}

\begin{remark}
	This definition is independent of a Riemannian metric as, at critical points, the Hessian coincides with $d^2 h$:
	\[\He h(X, Y)= d^2 h(X, Y)= Y(X h).\]
\end{remark}
\begin{definition}
	Let $m\in M$ be a critical point of a smooth function $h: M\mapsto \mathbb{R}$. Its index is defined to be the index of $\He h$ at that point. That is, the index counts the number of negative eigenvalues of the symmetric matrix $\He h$.
\end{definition}
R. Bott did pioneering work \cite{bott54} to extend the notion leading to a broader range of applications. In modern language, we recall the following. 

\begin{definition} A smooth submanifold $N\subset M$ is called a non-degenerate critical sub-manifold of $h$ if the following holds:
	\begin{itemize}
		\item $N$ is connected;
		\item $N$ is a subset of the set of critical points of $h$;
		\item $\forall p\in N$ we have $T_p N= \text{Ker} (\He h_{p})$. That is, $\He h_{p}(X, Y)=0 ~~\forall Y\in T_p M \iff X\in T_pN$.
	\end{itemize} 
	The function $h$ is called a Morse-Bott function if its critical set consists of non-degenerate critical submanifolds.
\end{definition}

Next we recall the fundamentals of Hamiltonian dynamics. Our references are \cite{Cannas01_book, BF_book04, VNS07, PRV15, PV09}. Let's start with some algebra. 
\subsection{Symplectic Algebra}
A $2n$-dimensional vector space $V$ is called symplectic if it is endowed with a non-degenerate bilinear skew-symmetric $2$-form $\omega$. $\omega$ is called a symplectic form or symplectic structure of $V$. If we fix a Darboux coordinate system $\{e_1,... e_{2n}\}$, then $\Omega$ denotes the matrix associated with $\omega$; that is, $\Omega_{ij}=\omega(e_i, e_j)$. $\Omega$ is said to be in a canonical form if 
 \[\Omega=\begin{bmatrix}
 	0 & \id \\
 	\id & 0
 \end{bmatrix}\] 
 
 Given a symplectic form $\omega$ on a $2n$-dimensional vector space $V$, a linear transformation $A: V\mapsto V$ is called symplectic if it preserves $\omega$:
\[ \omega(X, Y)= \omega(AX, AY).\]
The set of all symplectic transformations forms a Lie group, denoted by $\text{Sp}(2n, \mathbb{R})$. Here are a few facts about this group.
\begin{itemize}
	\item Symplectic transformations are uni-modular; that is, \[\text{det}(A)=1~~ \forall A\in \text{Sp}(2n, \mathbb{R}).\]
	\item The characteristic polynomial $P(\lambda)=\text{det}(A-\lambda \id)$ satisfies:
	\[P(\lambda)=\lambda^{2n}P(\frac{1}{\lambda}).\]
	\item $\text{Sp}(2n, \mathbb{R})$ is a non-compact Lie group of dimension $n(2n+1)$. 
	\item Topologically, $\text{Sp}(2n, \mathbb{R})$ is diffeomorphic to the direct product of the unitary group $U(n)$ (of dimension $n^2$) and the Euclidean vector space $\mathbb{R}^{n(n+1)}$. 
	\item $\text{Sp}(2n, \mathbb{R})$ is path connected but not simply connected; its fundamental group is $\mathbb{Z}$. 
\end{itemize}  

The corresponding Lie algebra $\mathfrak{sp}(2n, \mathbb{R})$ consists of matrices $A$ satisfying the relation 
\[A^T \omega+ \omega A=0.\]
If the basis is canonical, 
\[\Omega= \begin{bmatrix}
	0 & \id \\
	-\id & 0
	\end{bmatrix},
	\]
	then 
\[A= \begin{bmatrix}
	A_1 & A_2 \\
	A_3 & -A_1^T
\end{bmatrix},
\]
for $A_1$ is arbitrary while $A_2$ and $A_3$ are symmetric. Here are a few facts:
\begin{itemize}
	\item $\mathfrak{sp}(2n, \mathbb{R})$ is isomorphic to the algebra of homogeneous quadratic polynomials with Poisson brackets. 
	\item There is a Jordan decomposition: any homogeneous quadratic polynomial is a summation of a semisimple part (conjugate to a diagonalized matrix) and  a nilpotent part:
	 \[H_2=H_{ss}+H_{nil}.\]  
	\end{itemize}

\subsection{Integrable Systems}
A symplectic manifold is a differentiable manifold $M$ equipped with a smooth closed $2$-form $\omega$ such that, for each $m \in M$, $(T_mM, \omega_{\mid_m})$ is a symplectic vector space. A vector field $X$ on $M$ is then called symplectic if it generates a flow preserving $\omega$. Equivalently, $\imath_X\omega$ is closed.

\begin{definition} A vector field $X$ is called Halmiltonian if $\imath_X\omega$ is exact. A primitive $H$, $\imath_X\omega= dH$, is called a Halmiltonian function of $X$. 
\end{definition} 
	
Alternatively, one starts with a symplectic manifold $(M, \omega)$ and a function $H$. $\omega$ induces an isomorphism between vector fields and one-forms. Thus, by non-dengenacy of the symplectic form $\omega$, there is a unique vector field $X_H$ on $M$ such that
\[dH= \imath_{X_H}\omega.\] 
It is immediate to check that the flow generated by $X_H$  preserves the symplectic structure. In the presence of a Riemannian metric and an almost complex structure such that $\omega=g(\cdot, J\cdot)$, then we have 
\[ X_H= J\nabla H.\]

\begin{definition}
	The triple $(M, \omega, H)$ is called a Hamiltonian system.
\end{definition}

Next, we recall the definition of a Poisson bracket.
\begin{definition}
	The Poisson bracket of two functions $f_1, f_2\in C^\infty (M, \mathbb{R})$ is given by 
	\[\{f_1, f_2\}:= \omega(X_{f_1}, X_{f_2}). \]
\end{definition}

In addition to bi-linearity and skew-symmetry, the Poisson bracket satisfies the Jacobi's identity and Leibniz rule. Thus, the operator mapping a function $H$ to $X_H$ is a Lie algebra homomorphism. As a result, the following identity holds \[X_{\{f_1, f_2\}}=- [X_{f_1}, X_{f_2}].\]

Using a Darboux coordinate system, which has $\omega= \sum_{j=1}^n dx_j\wedge dy_j$, the Poisson bracket can be computed as follows:
\[\{f_1, f_2\}= \sum_{i=1}^n \frac{df_1}{dx_j} \frac{df_2}{dy_j}- \frac{df_1}{dy_j} \frac{df_2}{dx_j}\]

\begin{definition} A function F is a first integral of motion for the Hamiltonian system $(M, \omega, H)$ if $\{F, H\}=0$.
\end{definition}

\begin{remark}
	Note that $\{F, H\}=0$ if and only if $F$ is constant along integral curves of $X_H$.
\end{remark}

Two functions are called functional independent if their differentials are linearly independent on $M$ almost everywhere. It means they might be linearly dependent on a set of measure zero. indeed, many interesting phenomenon of a system depend on that set of singularities. 
\begin{definition}
	A Hamiltonian system $(M, \omega, H)$ is integrable if there are $n=\frac{1}{2}\text{dim}(M)$ functionally independent integrals of motion $f_1=H, f_2, ..., f_n$ such that, for all $1\leq i, j\leq n$,
	\[ \{f_i, f_j\}=0.\] 
	Furthermore, the system is called completely integrable if each vector field $X_{f_i}$ is complete.  
\end{definition}

The function\[\Phi=(f_1,..., f_n): M\mapsto \mathbb{R}^n\] is called the moment map. For a completely integrable system, each connected component of the pre-image of a regular value, $\Phi^{-1}(c)$, is a Lagrangian submanifold. Indeed, it must be homogeneous of the form $\mathbb{R}^{n-k}\times \mathbb{T}^k$, where $ \mathbb{T}^k$ is a $k$-dimensional torus. Any compact connected component is then a $n$-dimensional torus. 

Furthermore, for a completely integrable system with a proper moment map, the well-known Liouville-Arnold theorem guarantees there is a local torus $T^n$-action preserving the system. More precisely, around each regular point, there is an open neighborhood $U$ and a diffeomorphism $\chi: F(U)\mapsto \mathbb{R}^n$ such that
\begin{itemize}
	\item For $V=(X\circ \Phi)(U)$  there is a diffeomorphism $(\phi, \theta): U\mapsto V\times \mathbb{T}^n$ 
	\item The diffeomorphism provides natural action and angle coordinates $\{\phi_1,..., \phi_n, \theta_1, ..., \theta_n \}$, which form a Darboux chart. That is, 
\[\omega= \sum_j d\phi_j\wedge d\theta_j.\]
\item The action variables $\{\phi_1,...\phi_n\}$ are functions of the integrals of motion $f_1,... f_n$. 
\item The flow of each $X_{f_i}$ is linear with respect to this coordinate system. 
\end{itemize}

The existence of a global torus action is highly non-trivial. This is equivalent to extending the action across singular points and J. Duistermaat \cite{Duis80} first observed there is some monodromy. The monodromy is described in terms of a covering group such that its non-triviality is a topological obstruction to go from local to global. To understand such behavior, it is inevitable to study the set of singular points where the differential of $\Phi$ is not of the maximal rank.

\subsection{Singular Points of an Integrable System}
We follow \cite{BF_book04, VNS07, PV09, PRV15} closely here. Let $(M^{2n}, \omega)$ be an integrable system with a moment map:
\[\Phi=(f_1,... f_n): M\mapsto \mathbb{R}^n.\]
Generally, the singular set consists of points such that at which the rank of the differential of the moment map is not maximum. The set of singular points has a stratification and, in case the system is analytic, it is an analytic set. 

\begin{definition}
	A point $m\in M$ is called critical/singular with respect to $\Phi$ if 
	\[\text{rank} (d\Phi)(m)<n.\]
	Its image $\Phi(x)$ is called a critical value. 
\end{definition}

The general guideline is to study the linearized system. Towards that goal, $m$ is a critical point if and only if the gradients $\{df_i\}_{i=1}^n$ are linearly dependent at this point. Let's consider a linear combination yielding a function $h$ such that  $dh=0$ at $m$. The flow of $X_h$ generates a (local) one-parameter group of transformations fixing $m$. The differential at $m$ is generated by a linear operator $A_h$. This is called the linearization of the Hamiltonian vector field $X_h$ at a singular point.
\begin{remark}
	In the presence of a Riemannian metric, or an affine connection, the linear operator is exactly the derivation $A_h=A_{X_h}$.
\end{remark}
 Since the flow generated by $X_h$ preserves the symplectic structure, $A_{X_h}$ is a symplectic operator. Thus, it is as an element of the corresponding Lie algebra $\mathfrak{sp}(2n, \mathbb{R})$.  With a choice of a Darboux coordinate system around $m$, one observes that $A_{h}$ corresponds to the Hessian of $h$:
\[\Omega^{-1} d^2 h.\]
Here, $\Omega$ is the matrix corresponding to $\omega$ with respect to the choice of Darboux coordinates. Also, $d^2h$ makes sense as the Hessian since $dh(m)=0$.  

\begin{definition} Let $m$ be a singular point of the moment map $\Phi$ such that $\text{rank}(d\Phi)=i<n$. Let $K(m, \Phi)$ be the commutative subalgebra generated by the linearized operators as described above. $m$ is called non-degenerate if $K(m, \Phi)$ is a Cartan subalgebra in $\mathfrak{sp}(2(n-i), \mathbb{R})$. Consequently, a Hamiltonian system is called non-degenerate if the condition holds for each of its singular point. 
\end{definition} 
An element of a Cartan subalgebra is called regular (or generic) if its eigenvalues are all distinct. The type of the singularity then can be determined by the following procedure. For a regular element $H$ we consider:
\[\text{det}(\Omega^{-1})\text{det}(d^2H-\lambda \Omega)= \text{det}(\Omega^{-1} d^2H -\lambda \id).\]
The roots of the polynomial split into pairs: 
\begin{itemize}
	\item pairs of imaginary roots $i\alpha$, $-i\alpha$: elliptic;
	\item pairs of real roots $\beta, -\beta$: hyperbolic;
	\item quadruples of complex conjugate roots: $\alpha+i\beta, \alpha-i\beta, -\alpha+i\beta, -\alpha-i\beta$: focus-focus.
\end{itemize}

\begin{definition}
	A Cartan subalgebra $K\in \mathfrak{sp}(2j, \mathbb{R})$ is characterized by a triple $(j_1, j_2, j_3)$ which counts the number of roots in each category described above for a regular element and:
	\[j_1+j_2+2j_3= \text{dim}(K)=j.\]
\end{definition}
Indeed, the type of a singularity is tied to that of the corresponding Cartan sub-algebra.  
\begin{definition}
	A non-degenerate singular point $m$ is characterized by its rank and the triple $(j_1, j_2, j_3)$ of $K(m, \Phi)$.  
\end{definition}

For illustration, consider the simple example where $i=0$ and $n=1$. For the description below, we identity an element of  $\mathfrak{sp}(2, \mathbb{R})$ with a homogeneous quadratic polynomial and the singular point with the origin. There are two Cartan subalgebras:
\[
d^2(x^2+y^2)=\begin{bmatrix}
	2 & 0 \\
	0 & 2 
\end{bmatrix} (\text{elliptic}) ~~ \text{ and }~~~
d^2(xy)=\begin{bmatrix}
	0 & 1\\
	1 & 0
\end{bmatrix}(\text{hyperbolic}).~~
\]
For \[ \Omega=
\begin{bmatrix}
	0 & 1 \\
	-1 & 0 
\end{bmatrix},\]
it is immediate to check that the singularities correspond to elliptic and hyperolic, respectively. 
\subsection{Integrable Systems with Two Degree of Freedom}
In this paper, we are mostly interested in KGRS in real dimension four, which corresponds to an integrable system with two degree of freedom. It is immediate that the set of singular points, denoted by $\mathcal{S}$, is closed in $M$ and there is stratification:
\[\mathcal{S}= \mathcal{S}_0 \cup \mathcal{S}_1.\]
Here $\mathcal{S}_i$ denotes the set of singular points of rank $i$ for $i=0, 1$. 

The notions associated with a singular point can be reinterpreted as follow. $m\in \mathcal{S}_1$ is non-degenerate if and only if, supposed that $df_1\neq \vec{0}$, $f_2$ is non-degenerate at $m$ when it is considered as a function on the level set $f_1^{-1}(f_1(m))$. In other words, for a non-degenerate system, $f_2$ is Morse-Bott on each level set of $f_1$. \\

For $m\in \mathcal{S}_0$, recall that a commutative subalgebra in $\mathfrak{sp}(4, \mathbb{R})$ is Cartan if and only if it is two dimensional and contains an element whose eigenvalues are distinct. Indeed, there is a full classification. 

\begin{theorem}
	\label{typeofsing}
	\cite[Theorem 1.3]{BF_book04}
	Let $K$ be a Cartan subalgebras of $\mathfrak{sp}(4, \mathbb{R})$. Then it is conjugate to one of four Cartain subalgebras listed below
	\[
	\begin{bmatrix}
		0 & 0 & -A & 0\\
		0 & 0 & 0 & -B\\
		A & 0 & 0 & 0\\
		0 & B & 0 & 0
	\end{bmatrix},~~
	\begin{bmatrix}
		-A & 0 & 0 & 0\\
		0 & 0 & 0 & -B\\
		0 & 0 & A & 0\\
		0 & B & 0 & 0
	\end{bmatrix},~~
	\begin{bmatrix}
		-A & 0 & 0 & 0\\
		0 & -B & 0 & 0\\
		0 & 0 & A & 0\\
		0 & 0 & 0 & B
	\end{bmatrix},~~
	\begin{bmatrix}
		-A & -B & 0 & 0\\
		B & -A & 0 & 0\\
		0 & 0 & A & -B\\
		0 & 0 & B & A
	\end{bmatrix}.
	\]
	They correspond to elliptic-elliptic, elliptic-hyperbolic, hyperbolic-hyperbolic, focus-focus types in that order. 
\end{theorem}

Recall that the type of a singular point is characterized by its rank and the triple associated with its corresponding Cartan sub-algebra. Therefore, the list above corresponds to the following triples:
\begin{itemize}
	\item Elliptic-elliptic: $(2, 0, 0)$
	\item Elliptic-hyperbolic: $(1, 1, 0)$
	\item Hyperbolic-hyperbolic: $(0, 2, 0)$
	\item Focus-focus: $(0, 0, 1)$
\end{itemize}

It turns out that many properties of a general interable system are captured by the corresponding linearized system. A fundamental result in this spirit is the following normal form characterization. 

\begin{theorem}
	\label{normalform}
	Let $m$ be a non-degenerate singularity of a momentem map $\Phi=(f_1, f_2)$ on a $4$-dimensional symplectic manifold. Then there exist local symplectic coordinates $\{x_1, x_2, y_1, y_2\}$ in a neighborhood of $m$ such that 
	\[\{f_i, q_j\}=0 ~~ \forall i, j\in \{1, 2\}, \]
	where 
	\begin{enumerate}
		\item If $\text{rank}(m)=0$ then $q_i$ is one of the following:
		\begin{itemize}
			\item $q_i=\frac{x_i^2+y_i^2}{2}$ (elliptic bloc)
			\item $q_i= x_iy_i$ (hyperbolic bloc)
			\item $\begin{cases}
				q_1 &=x_1 y_2-x_2y_1,\\
				q_2 &=x_1y_1+x_2y_2.
			\end{cases}$ (focus-focus bloc)
		\end{itemize}
		\item if $\text{rank}(m)=1$ then $q_1=y_1$ (non-singular) and $q_2$ is one of the following:
		\begin{itemize}
			\item $q_2=\frac{x_2^2+y_2^2}{2}$ (elliptic)
			\item $q_2=x_2 y_2$ (hyperbolic)
		\end{itemize}
	\end{enumerate}
	In case there is no hyperbolic bloc, then there is a local diffeomorphism, $\Psi$, of $(\mathbb{R}^2, 0)$ such that
	\[\Phi= \Psi\circ (q_1, q_2).\]
\end{theorem}

\section{Integrable}
\label{integrable}
In this section, we will show that a KGRS is indeed an integrable Hamiltonian system. The following is well-known to the experts; for example, see \cite{CST09note,CDSexpandshriking19}. We'll repeat a proof for completeness as it is crucial to our development.
\begin{lemma}
	\label{fMorseBott}
	Let $(M, g, J, f, \lambda)$ be a KGRS. Then, $f$ is a Morse-Bott function with even indices. 
\end{lemma}

\begin{proof}
	Since $J\nabla f$ is a Killing vector field, it generates an one-parameter family of isometries $\varphi_t$. The set of critical points of $f$ coincides with the set of fixed points of $\Phi_t$. By Kobayashi \cite{kobay58}, each connected component is a totally geodesic submanifold of an even co-dimension. 
	
	Moreover, the Hessian of $f$ is intrinsically related to the differential of $\varphi_t$ (see [Prop.VI.4.1]\cite{KNvolumeI96}). Since isometries commute with the exponential map, the differential is non-identity in each normal direction. Thus, the Hessian of $f$ is non-degenerate on the normal bundle to each connected critical component. Consequently, $f$ is Morse-Bott.
	
	For the index, recall that $\He f$ commutes with $J$. Thus, each eigenvalue is of multiplicity two. Therefore, each index is even.  
\end{proof}
\begin{corollary}
	\label{connectedlevel}
	Let $(M, g, J, f, \lambda)$ be a KGRS. Supposed that $f$ is proper, then each level set is connected. Consequently, each local minimum or maximum must be a global one.
\end{corollary}
\begin{proof}
	It is shown by \cite[Prop 4.5]{PRV15}, the level set of a proper Morse-Bott function with index and coindex different from one is connected (the compact case was observed by \cite{Atiyah82}; see also \cite[Lemma 3.46]{Nico07book} for an another proof). The conclusion then follows from Lemma \ref{fMorseBott} and the fact that the dimension of the manifold is even. 
	
	The second statement follows from \cite[Lemma 5.1]{PRV15}. 
\end{proof}
\begin{lemma}
	\label{poissonfS} Let $(M, g, J, \omega, f, \lambda)$ be a KGRS. Then,
	\[\omega(\nabla f, \nabla \SS)=0.\]
\end{lemma}
\begin{proof}
	We compute
	\begin{align*}
		\omega(\nabla f, \nabla \SS) &= 2\omega(\nabla f, \Rc(\nabla f))\\
		&= 2g(\nabla f, J\Rc(\nabla f)),\\
		&= -2 g(J(\nabla f), \Rc(\nabla f))\\
		&=- 2\Rc(\nabla f, J\nabla f). 
	\end{align*}
Since $(M, g, J, \omega)$ is K\"{a}hler, $\Rc$ commutes with $T$. Thus, 
\[\Rc(\nabla f, J\nabla f)=\Rc(J\nabla f, J^2 \nabla f)=-\Rc(\nabla f, J\nabla f)=0.\]

\end{proof}

\begin{lemma}
	\label{completevf} Let $(M, g, J, \omega, f, \lambda)$ be a KGRS.
	If $f$ is proper then $J\nabla \SS$ and $J\nabla f$ are complete. 
\end{lemma}
\begin{proof}By Lemma \ref{poissonfS}, 
	\[g(\nabla f, J\nabla \SS)=\omega (\nabla f, \nabla \SS)=0.\]
	Thus, $J\nabla \SS$ is perpendicular to  $\nabla f$. Consequently, given $p\in M$, the flow generated by either $J\nabla \SS$ or $J\nabla f$ starting at $p$ is in a connected component of the level set of $f^{-1}(f(p))$. Such a level set is compact since $f$ is proper. The result then follows. 
\end{proof}

We are ready to prove the integrability of a KGRS. 

\begin{proof}[Proof of Theorem \ref{main1}]
	First, $f$ is Morse-Bott by Lemma \ref{fMorseBott}. The connectedness of its level sets follows from Corollary \ref{connectedlevel}. Then, one recalls that, in the presence of a K\"{a}hler structure,
	\[X_f=J\nabla f. \]
	We compute, using Lemma \ref{poissonfS}
	\begin{align*}
		\{f, \SS\} &=\omega (X_f, X_{\SS})\\
		&= \omega(J\nabla f, J\nabla \SS)\\
		&=\omega(\nabla f, \nabla \SS)=0.
	\end{align*}
	Since the dimension of $M$ is four, the result follows. 
\end{proof}
\begin{remark}
	For a KGRS in a higher dimension, it is expected that $f$ and $\SS$ form part of the moment map for a suitable integrable system. 
\end{remark}

\begin{proof}[Proof of Corollary \ref{cor1}]
	Recall that, on a gradient Ricci soliton, $g$ and $f$ are real analytic \cite{bando87, Ilocal, DW11coho}. Thus, the function 
	\[|\nabla f\wedge \nabla \SS|=|\nabla f|^2|\nabla \SS|^2- g(\nabla f, \nabla \SS)^2\] is real analytic. Since $M$ is connected, it is either zero everywhere or there is a dense subset of $M$ on which the function is non-zero. That is, either $\nabla f$ is parallel to $\nabla \SS$ everywhere or $f$ and $\SS$ are functionally independent. By \cite{tran23contact}, the first case happens if and only if either $(M, g, f, J)$ is 
\begin{itemize} 
	\item a product of a constant curvature surface with a $2D$ K\"{a}hler GRS, or
	\item  of cohomogeneity one, $f$ is invariant by its action, and each principal orbit is a connected deformed homogeneous Sasakian structure. 
\end{itemize}
It is straightforward to check that such a soliton is isometrically toric and, thus, completely integrable. The second case follows from Theorem \ref{main1} and Lemma \ref{completevf}.

\end{proof}

\section{The Hessian of $\SS$}
\label{aboutHessS}
In this section, we will consider a KGRS $(M, g, J, \omega, f, \lambda)$ and derive a general formula for the Hessian of the scalar curvature. We start with a series of useful results. Let $H_f$ be the $(1, 1)$ tensor field associated with $\He f$. That is, 
\[ g(H_f X, Y)= \He f(X, Y)= g(\nabla_X\nabla f, Y).\]
Note that $H_f$ commutes with $J$ which is parallel. We also use $\He^2 f$ to denote the $(2,0)$ tensor dual to $H_f^2$. 
\begin{lemma} We have
	\begin{align*}
		0 &= \He f(\nabla f, J\nabla \SS).
	\end{align*}
\end{lemma}
\begin{proof}
	By Lemma \ref{poissonfS}, 
	\[g(J\nabla \SS, \nabla f)=0.\]
	In addition, as $\Rc$ commutes with $J$,
	\begin{align*}
		\Rc(\nabla f, J\nabla \SS) &= -\Rc(J\nabla f, \nabla \SS),\\
		&= -g(\Rc(J\nabla f), \nabla \SS),\\
		&= 2 g(J\Rc(\nabla f), \Rc(\nabla f))=0.
	\end{align*}
	The result then follows from the soliton equation $\He f+\Rc =\lambda g$. 
\end{proof}
Recall, for a vector field $X$, $A_X$ is the derivation
\[A_X Y=\lie_X Y-\nabla_X Y=-\nabla_Y X.\]

\begin{lemma} We have
	\label{AvsHess}
	\[A_{J\nabla f}=-J\circ H_f.\]
\end{lemma}
\begin{proof}
We compute:
\begin{align*}
	A_{J(\nabla f)}Y &= -\nabla_Y (J\nabla f)=-J \nabla_Y \nabla f=-J H_f(Y).
\end{align*}
\end{proof}
\begin{lemma} We have
	\label{covHf}
	\[\nabla_X H_f=-J\circ \RR(X, J\nabla f)\]
\end{lemma}
\begin{proof}
	Recall that $J\nabla f$ is a Killing vector field. Thus, 
	\[\nabla_X (A_{J\nabla f})=\RR(J\nabla f, X).\]
	Since $J$ is parallel, we compute, 
	\begin{align*}
		-R(X, J\nabla f)&= \nabla_X (A_{J\nabla f})=-\nabla_X (J\circ H_f)\\
		&=-(\nabla_X J)H_f -J\circ \nabla_X H_f= -J\circ \nabla_X H_f.
	\end{align*}
\end{proof}
\begin{proposition} \label{hessS} 
	Let $(M, g, J, \omega, f, \lambda)$ be a KGRS. The Hessian of the scalar curvature is computed as follows:
		\begin{align*}
		H_\SS (X) &= 2\lambda H_f(X)-2H_f^2(X)+2 J\circ \RR(X, J\nabla f)(\nabla f),\\
		\He \SS (X, Y)&=2\lambda \He f(X, Y)-2 \He^2 f(X,Y)-2 \RR(X, J\nabla f, Y, J\nabla f).
	\end{align*}
	
\end{proposition}
\begin{proof}
	Recall \[\nabla \SS=2\Rc(\nabla f)=2(\lambda g-H_f)(\nabla f).\]
	Thus,
	\begin{align*}
		\nabla_X \nabla\SS &=2 \nabla_X (\lambda \nabla f- H_f(\nabla_f))\\
		&= 2\lambda H_f(X)- 2(\nabla_X H_f)(\nabla f)-2H_f(\nabla_X \nabla_f)\\
		&= 2\lambda H_f(X)-2H_f^2(X)+2(J\circ \RR(X, J\nabla f))(\nabla f).
	\end{align*}
	Consequently,
	\begin{align*}
		\He \SS(X, Y)&= g(\nabla_X \nabla\SS, Y)\\
		&=2\lambda \He f(X, Y)-2\He f(H_f X, Y)-2 g(\RR(X, J\nabla f)(\nabla f), JY).
	\end{align*}
\end{proof}
In particular, 
\[H_S J\nabla f=2\lambda H_f(J\nabla f)-2H_f^2(J\nabla f)=J H_f(\nabla \SS).\]

The following calculation might be of independent interest. Recall \cite[Theorem 3.2.2]{pe06book}, in a small neighborhood of a regular point, we have
\begin{align*}
	(\nabla_{\nabla f}H_f)(X)+H_f^2(X)-\nabla_X (H_f(\nabla f))&=-\RR(X, \nabla f)\nabla f;\\
	\nabla_{\nabla f}\He f+ \He^2 f-\He(\frac{1}{2}|\nabla f|^2)&=-\RR(\cdot, \nabla f, \nabla f, \cdot),\\
	\lie_{\nabla f}\He f &= \nabla_{\nabla f}\He f+2\He^2 f. 
\end{align*}
It is noted that covariant derivation commutes with type changes. 

\begin{lemma} $\nabla_X H_f$ is $J$-invariant and, for an unit vector field $E_i\perp \text{span}(\nabla f, J\nabla f)$,  
	\[(\nabla_{\nabla f}\He f)(E_i, E_i)=|\nabla f|^2(-\Rc(E_i, E_i)+\text{sect}(E_i, JE_i)).\]
\end{lemma}
\begin{proof}
	We have, by Lemma \ref{covHf},
	\begin{align*}
			(\nabla_{X}\He f)( JY, JZ) &=g((\nabla_X H_f) JY, JZ)\\
			&= g(-J\circ \RR(X, J\nabla f)JY, JZ)\\
			&= -\RR(X, \nabla f, JY, Z)=\RR(X, \nabla f, Y, JZ),\\
			&=(\nabla_{X}\He f)( Y, Z).
	\end{align*}
Using the first Bianchi's identity yields,
	\begin{align*}
		(\nabla_{\nabla f}\He f) (E_i, E_i) &= \RR(\nabla f, J\nabla f, E_i, JE_i)\\
		&= -\RR(J\nabla f, E_i, \nabla f, JE_i)-\RR(E_i, \nabla f, J\nabla f, JE_i)\\
		&=-\RR(E_i, J\nabla f, J\nabla f, E_i)-\RR(E_i, \nabla f, \nabla f, E_i).
		\end{align*}
	Thus, for unit vector field $E_i\perp \nabla f, J\nabla f$
	\begin{align*} (\nabla_{\nabla f}\He f) (E_i, E_i) &=-\Rc(\nabla f, \nabla f)+\text{sect}(J\nabla f, \nabla f),\\
		&=|\nabla f|^2(-\Rc(E_i, E_i)+\text{sect}(E_i, JE_i)).
	\end{align*}
\end{proof}

\section{Geodesics}
\label{sectiongeo}
In this section, on a KGRS $(M, g, J, \omega, f, \lambda)$, we consider $\gamma: I=[a, b]\mapsto M$, an integral curve of $\nabla f$ such that $\vec{0}\neq \nabla f \parallel \nabla \SS$ at each point on $\gamma$. This curve and the flow generated by $J\nabla f$ will form a two dimensional manifold of special interest. It will be crucial to our analysis of singularities in the next section. We start with simple observations.

\begin{lemma}
	\label{geodesic}
	$\gamma$ is a reparamertrized geodesic.
\end{lemma}
\begin{proof}
	From the soliton equations:
	\begin{align*}
		\Rc+ \He f &= \lambda g,\\
		\Rc(\nabla f) &= \frac{1}{2}\nabla \SS,
	\end{align*}
it follows that $\nabla f\parallel \nabla \SS$ implies $\nabla f$ is an eigenvector of $\He f$. Then it is readily verified that a reparametrization of $\gamma$ is a geodesic. 
	
\end{proof}
Lemma \ref{geodesic} allows one to reparametrize $\gamma$ by arc length $s$ such that $f\circ \gamma =\id$ and $|\gamma'(s)|=1$. Let $\phi_{\theta}(\cdot)$ denote the flow  generated by $J\nabla f$. For sufficiently small $\epsilon$ let $N$ be the iamge of the map $\phi: \gamma \times (-\epsilon, \epsilon)\mapsto M$ via the identification, for any $p\in \gamma$,
\[\phi(p, \theta)=\phi_{\theta}(p).\]

\begin{lemma}
	\label{totalgeo}
	$N$ is a totally geodesic submanifold.
\end{lemma}

\begin{proof}
	First, the flows starting at different points on $\gamma$ of $J\nabla f$ are mutually distinct as they correspond to different level sets of $f$. Since the flows of $\nabla f$ and $J\nabla f$ commute (which is, in turn, because $H_f$ and $J$ commutes), they form a natural open chart for $N$. Thus, $N$ is a submanifold.   
	
	Next, as $J\nabla f$ is a Killing vector field preserving $f$, the relation $\nabla f \parallel \nabla \SS$ holds at each point on $N$. By the same argument as in Lemma \ref{geodesic}, $\nabla f$ and, thus, $J\nabla f$ are eigenvectors of $\He f$ with the same eigenvalue at each point on $N$. 
	
	Indeed, $TN$ is just the sub-bundle spanned by $\nabla f$ and $J\nabla f$. Let $V$ be a vector field perpendicular to $\nabla f$ and $J\nabla f$. We have
	\begin{align*}
		g(\nabla_{\nabla f} \nabla f, V)=g(\nabla_{J\nabla f} \nabla f, V)=g(\nabla_{\nabla f} J\nabla f, V)=g(\nabla_{J\nabla f} J\nabla f, V)=0.
	\end{align*}
	Thus, the second fundamental form is totally vanishing and the result follows. 
\end{proof}

Let $f':= \nabla_{\gamma'(s)}f=\frac{\partial f}{\partial s}$ and we denote higher derivatives analogously. It is well-defined since $\gamma$ is a geodesic and covariant derivatives agree with ordinary ones. Since $J(\nabla f)$ is a Killing vector field preserving $f$, $|\nabla f|$ is independent of $\theta$.

\begin{lemma}
	\label{curK}
	The induced metric $(N, g_N)$ is a warped product:
	\[ g_K = ds^2+ (f')^2 d\theta^2.\]
	Then, the only non-trivial sectional curvature is
	\begin{align*}
		\RR(\partial_\theta,\partial_s)\partial_s &=-\frac{f'''}{f'}\partial_{\theta}
	\end{align*}
\end{lemma}
\begin{proof}
	By Lemma \ref{totalgeo},  $(N, g_N)$ is a totally geodesic submanifold. Thus the ambient Killing vector field $J\nabla f$ is generating an one-parameter group of isometries on $(N, g_N)$. Since $\theta$ is the coordinate associated with the flow generated by $J\nabla f$, we observe that the pushforward of vector $\partial_{\theta}$ has length 
	\[|J\nabla f|=|f'|.\]
 	Thus, the first statement follows. The second follows immediately as the the sectional curvature computation for a warped product is well-known; see, for example, \cite{pe06book}. 
\end{proof}

\begin{remark}
	Since $K$ is totally geodesic, the above calculation also gives the ambient sectional curvature. The sectional curvature for other planes at a point on $N$ will depend on the Jacobi equation. It will be investigated elsewhere. 
\end{remark}

Let $T^\perp N$ denote the normal bundle to $N$. Since $\He f$ is $J$-invariant, for any orthonormal frame  $\{E_1, E_2\}$ of $T^{\perp} N$,
\[ \mu:= \He f( E_i, E_i).\]


\begin{lemma}
	\label{hessSonK} On $N$, there is an orthonormal frame $\{\gamma', J\gamma', E_1, E_2 \}$ diagonalizing both $\He f$ and $\He \SS$ such that
\begin{align*}
	\He f(\gamma', \gamma')=\He f(J\gamma', J\gamma') &= f'';\\
	\He \SS (\gamma', \gamma') &= 2\lambda f''-2(f'')^2-2 f' f''',\\
	\He \SS(J\gamma', J\gamma') &= 2\lambda f''-2(f'')^2,\\ 
	\He f (E_i, E_i) &= \mu;\\
	\He \SS (E_i, E_i) &= 2\lambda \mu-2\mu^2+2(f')^2\RR(JE_i, \gamma', \gamma', JE_i),\\ 
\end{align*}
\end{lemma}
\begin{proof}
	By Lemma \ref{hessS}, 
\begin{align*}
	\He \SS (X, Y)&=2\lambda \He f(X, Y)-2\He^2 f(X, Y)-2 \RR(X, J\nabla f, Y, J\nabla f)\\
	&= 2\lambda \He f(X, Y)-2\He^2 f(X, Y)+2(f')^2 \RR(JX, \gamma', \gamma', JY).
\end{align*}
Lemma \ref{curK} allows one to compute $\RR(X, J\nabla f, Y, J\nabla f)$ for $X\in \text{span}\{\nabla f, J\nabla f\}=TN$. 

Next, one observes that, due to its algebraic properties, $\RR(\cdot, \gamma', \gamma', \cdot)$ is symmetric on $T^\perp N$. Thus, there is an orthonormal frame diagonalizing both $\He f$ and $\RR(\cdot, \gamma', \gamma', \cdot)$ simultaneously. The result then follows.

\end{proof}

\section{Analysis of Singularities}
\label{singularity}
In this section, we will consider $(M, g, J, f, \lambda)$, a KGRS in real dimension four. Accordingly, by Theorem \ref{main1}, $(M, \omega, f, \SS)$ is an integrable system. We will analyze singularities of this system and derive the proof of Theorem \ref{main2}. 

We first obtain a few observations for a general integrable system with a moment map:
\[\Phi=(f_1, f_2): M\mapsto \mathbb{R}^2.\]
It is always assumed that $\Phi$ is smooth. Recall that the set of singular points has a stratification. 
\[\mathcal{S}= \mathcal{S}_0 \cup \mathcal{S}_1.\]
Here $\mathcal{S}_i$ denotes the set of singular points of rank $i$ for $i=0, 1$. Let $V$ be some efficiently small neighborhood of a singular point $m$. The normal forms immediately imply the following:

\begin{lemma}
	\label{branchhyper}
	Let $m$ be a non-degenerate singular point in $\mathcal{S}$. 
	\begin{enumerate}[label=(\roman*)]
		\item Then the image $\Phi(V)$ is a disk unless it has an elliptic bloc. 
		\item If $m$ has an hyperbolic bloc, then there is an embedded line segment of critical values in the interior of $\Phi$.
		\item Let $m\in \mathcal{S}_0$, supposed that it is of type elliptic-elliptic, then it is connected to two branches of elliptic singular points $\mathcal{S}_1.$ 
	\end{enumerate}
	
\end{lemma}
\begin{proof}
	Recall that the Poisson bracket between two functions $\{F, H\}$ vanishes if and only if $F$ is constant on integral curves of $X_H$. By Theorem \ref{normalform}, there is a coordinate $(x_1, y_1, x_2, y_2)$ in a neighborhood of $m$ such that
	\begin{align*}
		\{f_i, q_j\} &=0.
	\end{align*}
	One obverses that each $f_i$ is invariant on each level set of each $q_j$. Thus, locally $f_i$ could be considered as a function of $(q_1, q_2)$. That is,
	\[ f_i=F_i (q_1, q_2).\]
	Then it is straightforward to compute the first and second derivative. For instance, we have
	\begin{align*}
		\frac{\partial f_i}{\partial x_j} &= \sum_k\frac{\partial F_i}{\partial q_k} \frac{\partial q_k}{\partial x_j}:= \sum_k F_{i,k} \delta_j^k q_{k, 1},\\
		\frac{\partial f_i}{\partial y_j} &= \sum_k\frac{\partial F_i}{\partial q_k} \frac{\partial q_k}{\partial y_j}:= \sum_k F_{i,k} \delta_j^k q_{k, 2}.
	\end{align*}
	Here, $\delta^k_j$ is the Delta notation reflecting the fact that $q_i$ is only a function of $x_i, y_i$. 
	The singularity condition means that, 
\[	\begin{cases}
		\text{for rank $1$, $df_1$ is a multiple of $df_2$ at the origin}\\
		\text{for rank $0$, $df_1=df_2=0$ at the origin.}\\
	\end{cases}
\]	
	(Part i) The image $\Phi(V)$ is not a disk if and only $m$ is a minimum or maximum point for at least one function $f_i$. Consequently, the Hessian of $f_i$ must be non-negative or non-positive. Comparing that observation with the non-degeneracy and the singularity condition for each type leads to the conclusion. 
	
	For instance, if $m$ is of type hyperbolic-hyperbolic then we have
	\begin{align*}
		\frac{\partial f_i}{\partial x_j} &=F_{i,j} y_j,\\
		\frac{\partial f_i}{\partial y_j} &= F_{i, j} x_j,\\
		\frac{\partial^2 f_i}{\partial^2 x_j} &=F_{i, jj} y_j^2,\\
		\frac{\partial^2 f_i}{\partial x_j\partial y_j} &=F_{i, jj} x_j y_j+ F_{i, j},\\
		\frac{\partial^2 f_i}{\partial^2 y_j} &=F_{i, jj} x_j^2.
	\end{align*}
	$m$ is a minimum or maximum for a function $f_i$ only if all eigenvalues of $d^2 f_i$ have the same sign. Given the form above, each $F_{i,j}= 0$, a contradiction to the non-degeneracy. Therefore, $\Phi(V)$ is a disk. \\
	
	(Part ii) Let's first consider $m\in \mathcal{S}_1$. Since non-degeneracy in this case is an open condition, there would be a continuous curve of singular values in $\Phi(V)$. Thus, the result follows from the previous part as $\Phi(V)$ is a disk.
	
	Next if $m\in \mathcal{S}_0$ then one observes that $m$ is connected to at least one branch of singularities of rank $1$ (in $\mathcal{S}_1$) such that the image of a small neighborhood of such a point is a disk. Thus, the result follows.\\
	
	(Part iii)  By the normal form described above, $\Phi(V)$ is a half-disk. Furthermore, its boundary is made of the images of $m$ and two small punctured disks in the $(x_1, y_1)$-plane ($x_2=y_2=0$) and the $(x_2, y_2)$ plane $(x_1=y_1=0$). Thus, each branch corresponds to a connected component of singularities in $\mathcal{S}_1$. Finally, one observes that each branch must be of elliptic type as the image if a small neighborhood around each point is not a disk and we recall part (i).

\end{proof}

\begin{remark}
	These observations are well-known to the experts; see, for example, \cite[Prop. 2.7]{VNS07} and \cite[Figure 7]{PRV15}. We provided a proof as these precise statements will be crucial for our analysis. 
\end{remark}

In case of a KGRS, thanks to the presence of the Riemannian metric, more can be said about the integrable system $(M, \omega, f, \SS)$. First, there is a simple observation.  
\begin{lemma}
	\label{lm1}
	A point $m$ is in $\mathcal{S}_0$ if and only if $\nabla f(m)=0$. Consequently, $m\in \mathcal{S}_1$ if and only if $\nabla f(m)\neq 0$ and $\nabla f(m)\parallel \nabla \SS(m)$. 
\end{lemma}
\begin{proof}
	One direction of the statement is obvious. The other follows from the fact that
	\[\Rc(\nabla f)=\frac{1}{2}\nabla \SS. \]
\end{proof}
\subsection{Analysis of $\mathcal{S}_0$}
The main goal of this section is to prove the following. 
\begin{theorem}
	\label{S0}
	Let $(M, g, J, \omega, f, \lambda)$ be a KGRS in real dimension four. With respect to the integrable system $(M, \omega, f, \SS)$, all non-degenerate singularities in $\mathcal{S}_0$ are isolated and of elliptic-elliptic type. If the system is non-degenerate then $f$ is Morse. 
\end{theorem}
The proof is divided into several lemmas. 
\begin{lemma}
	\label{chardege0}
	Let $m\in \mathcal{S}_0$ for the integrable system $(M, \omega, f, \SS)$. It is degenerate if and only if $\He f$ is either a multiple of the identity or has $0$ as an eigenvalue. If it is non-degenerate then it is of type elliptic-elliptic. 
\end{lemma}
\begin{proof}
	Since $m\in \mathcal{S}_0$, we have at this point $\nabla f=\nabla \SS=\vec{0}$. Consequently, $J\nabla f=\vec{0}$ at $m$. By Lemma \ref{hessS}, at $m$,
	\[\He \SS= 2\lambda \He f- \He^2 f.\] 
	Let $\{E_i\}_{i=1}^4$ be an orthonormal basis of $T_m M$ consisting of eigenvectors of $\He f$ such that
	\begin{align*}
		\He f(E_1, E_1) &= \He f(E_3, E_3)=\mu_1, \\
		\He f(E_2, E_2) &= \He f(E_4, E_4)=\mu_2.
	\end{align*}
	Thus, $\He \SS$ is diagonalized by the same basis and 
	\begin{align*}
		\He \SS(E_1, E_1) &= \He \SS(E_3, E_3)=2\lambda\mu_1-\mu_1^2,\\
		\He \SS(E_2, E_2) &= \He \SS(E_4, E_4)=2\lambda\mu_2-\mu_2^2.\end{align*}
	The results will follow from the following claims.\\
	
	\textbf{Claim 1:} If $\mu_1=\mu_2$ or $\mu_1=0\neq \mu_2$ then $m$ is degenerate. 
	
	\textit{Proof of Claim 1:} In this case, $\He f$ and $\He \SS$ are linearly dependent and the subalgebra generated will be of dimension one. Thus, it is degenerate.\\
	
	\textbf{Claim 2:} If $\{\mu_1, \mu_2, 0\}$ are mutual distinct then $m$ is non-degenerate of type elliptic-elliptic.
	
	\textit{Proof of Claim 2:} Recall that the matrix corresponding to $\omega$ is conventionally denoted
	\[\Omega=\begin{bmatrix}
		0 & 0 & 1 & 0\\
		0 & 0 & 0 & 1\\
		-1 & 0 & 0 & 0\\
		0 & -1 & 0 & 0
	\end{bmatrix}.\]
	A linear combination of $L=a\He f+b\He \SS$ is of the form 
	\begin{align*}
		L(E_1, E_1) &= \He f(E_3, E_3)=\mu_1(a+2b\lambda- b\mu_1), \\
		L(E_2, E_2) &= \He f(E_4, E_4)=\mu_2 (a+2b\lambda- b\mu_2).
	\end{align*}
	Since $\{\mu_1, \mu_2, 0\}$ are mutual distinct, $\He f$ and $\He S$ generate a two dimensional subalgebra. Furthermore, a generic element of that subalgebra, $\Omega^{-1}(a \He f+ b\He \SS)$, is of the form, for $A\neq B$,
	\[\begin{bmatrix}
		0 & 0 & -A & 0\\
		0 & 0 & 0 & -B\\
		A & 0 & 0 & 0\\
		0 & B & 0 & 0
	\end{bmatrix}.\]
	Checking the table in Theorem \ref{typeofsing} concludes that the singularity is of type elliptic-elliptic.   
\end{proof}

\begin{lemma}
	\label{S0isolated}
	Let $(M, g, J, \omega, f, \lambda)$ be a KGRS in real dimension four. With respect to the integrable system $(M, \omega, f, \SS)$, each non-degenerate point in $\mathcal{S}_0$ is isolated.
\end{lemma}

\begin{proof}
	We proceed by contradiction. By Lemma \ref{lm1}, $\mathcal{S}_0$ coincides with the singular set of $f$ which, in turn, coincides with the set of fixed points of isometries generated by $J\nabla f$. By \cite{kobay58}, each connected component is a totally geodesic submanifold of an even codimension. Thus, if $m$ is not an isolated point, then it belongs to a surface $\Sigma$ of dimension two on which $\nabla f_{\mid_\Sigma}=0$.  
	
	Thus, for vectors $X, Y$ tangential to $\Sigma$,
	\[ \He f(X, Y)= XY(f)-df(\nabla_X Y)=0.\]
	Since $\He f$ commutes with $J$, $m \in \Sigma$, $\He f$ has two eigenvalues $0$ and $\mu\neq 0$, each of multiplicity two. By Prop. \ref{fMorseBott}, $\mu_1\neq 0$. Lemma \ref{chardege0} then implies that $m$ is degenerate, a contradiction. 
	
\end{proof}
\begin{corollary}
	\label{fMorse}
		Let $(M, g, J, \omega, f, \lambda)$ be a KGRS in real dimension four. Supposed that $(M, \omega, f, \SS)$ is non-denegerate as an integrable system then $f$ is Morse. 
\end{corollary}
\begin{proof}
	By Prop. \ref{fMorseBott}, $f$ is Morse-Bott. By Lemma \ref{S0isolated} and the non-degeneracy hypothesis, each critical point is isolated. Thus, the result follows.  
\end{proof}

\begin{proof}[Proof of Theorem \ref{S0}] It follows from Lemmas \ref{chardege0} and \ref{S0isolated} and Corollary \ref{fMorse}.  
\end{proof}

\subsection{Analysis of $\mathcal{S}_1$}
The main goal of this section is to prove the following.
	\begin{theorem}
		\label{S1elliptic}
		 Let $(M, g, J, \omega, f, \lambda)$ be a complete KGRS. Supposed that the integrable system is non-degenerate and $f$ is proper and bounded below or above, then rank 1 singularities are of elliptic type. 
	\end{theorem}
It is interesting to note that the proof will not rely on the algebraic classification but, rather, a somewhat global argument. First, one recalls that the non-degenerate part of $\mathcal{S}_1$ is a two-dimensional manifold. 

\begin{lemma}
	\label{tangentspace}
	The set of non-degenerate singular points in $\mathcal{S}_1$ is a two-dimensional manifold and the tangent space at each point is  $\text{Span}(\nabla f, J\nabla f)$.
\end{lemma}
\begin{proof}
	Let $m\in \mathcal{S}_1$ be non-degenerate. By Lemma \ref{lm1}, $\nabla f\neq 0$. One then chooses a Darboux coordinate $\{x_1, y_1, x_2, y_2\}$ in a neighborhood of $m$ such that 
	\[f=x_1,~~ J\nabla f\parallel \frac{\partial}{\partial y_1}.\]
	Then, since $\{f, \SS\}=0$, $\SS$ is a function of $x_1, x_2, y_2$. Thus, the singularity is characterized by the equation
	\begin{equation*}
	\frac{\partial \SS}{\partial x_2}=0=\frac{\partial \SS}{\partial y_2}.
	\end{equation*}
	The non-degeneracy is equivalent to the non-degeneracy of the following matrix
	\[\begin{bmatrix}
		\frac{\partial^2 \SS}{\partial^2 x_2} & \frac{\partial^2 \SS}{\partial x_2 \partial y_2}\\
		\frac{\partial^2 \SS}{\partial x_2 \partial y_2} & \frac{\partial^2 \SS}{\partial^2 \partial y_2}
	\end{bmatrix}.\] 
	We then apply the implicit function theorem for the map 
	\[F: M\mapsto \mathbb{R}^2, ~~ x\mapsto (\frac{\partial \SS}{\partial x_2}\mid_x, \frac{\partial \SS}{\partial y_2}\mid_x ). \]
	Thus, the singularity equation describes a locally smooth two-dimensional submanifold in $M$. In addition, the non-degeneracy is an open condition and, thus, each point in a small neighborhood of that submanifold is non-degenerate. Finally, the implicit function theorem implies that the tangent space at each point of the submanifold is spanned by the push-forwards of vector fields $\frac{\partial}{\partial x_1}$ and $\frac{\partial}{\partial y_1}$. Due to our choice of the Darboux coordinate, this is exactly the span of $\nabla f$ and $J\nabla f$. 
\end{proof}
\begin{remark} The argument above is an adaptation of \cite[Prop 1.16]{BF_book04}. 
\end{remark}

Furthermore, since the set of all singular points $\mathcal{S}$ is closed, any accumulation point of $\mathcal{S}_1$ must be in $\mathcal{S}$. Indeed, since the non-degenerate part of $\mathcal{S}_0$ consists of only elliptic-elliptic types, we have the following. Let $\overline{\mathcal{S}_1}$ be the closure of $\mathcal{S}_1$

\begin{lemma}
	\label{limitS0}
	Supposed that $(M, \omega, f, \SS)$ is non-degenerate then 
	\[\overline{\mathcal{S}_1}=\mathcal{S}= \mathcal{S}_0 \cup \mathcal{S}_1.\]
\end{lemma}
\begin{proof}
	One direction follows from the fact that $\mathcal{S}$ is closed. For the other, consider a point $m\in \mathcal{S}_0$. By Theorem \ref{S0}, it is of type elliptic-elliptic. By Lemma \ref{branchhyper}, it is connected to at least two branches of elliptic singular points in $\mathcal{S}_1$. Thus, $m\in \overline{\mathcal{S}_1}$ and the statement follows. 
\end{proof}

\begin{lemma}
	\label{fundlem} Let $(M, g, J, \omega, f, \lambda)$ be a complete KGRS. Let $m\in \mathcal{S}_0$ and supposed that it is connected to a branch of non-degenerate singularities in $\mathcal{S}_1$. Then, there is a tangent vector $V$ at $T_mM$ such that 
	\begin{itemize}
		\item $V$ is an eigenvector of $\He f$ at $m$. 
		\item The two dimensional submanifold corresponding to the branch of singularities is the image of the exponential map of a punctured domain in the plane determined by $V, JV$. 
	\end{itemize} 
\end{lemma}
\begin{proof}
	By Lemma \ref{tangentspace}, the branch of non-degenerate singularities in $\mathcal{S}_1$ corresponds to a connected component of a two dimensional submanifold in $\mathcal{S}_1$. Denote this component by $N$. Then, at each point in $N$, 
	\[0\neq \nabla f \parallel \nabla \SS.\]
	
	Furthermore, since the tangent space of $N$ is spanned by $\nabla f$ and $J\nabla f$ by Lemma \ref{tangentspace}, each integral curve of $\nabla f$ starting at a point in $N$ stays in $N$. By Lemma \ref{geodesic}, each is a reparamertrized geodesic segment and the induced metric on $K$ has a warped product structure.
	
	Thus, $m$ is connected to $N$ only if it is an accumulation point for each of these geodesic segment. Denote an arbitrary geodesic segment by $\gamma$. Since the manifold is geodesically complete and $N$ is totally geodesic by Lemma \ref{totalgeo}, $\gamma$ can be extended through $m$. Lemma \ref{hessSonK} implies that $\gamma'$ is an eigenvector of $\He f$ at each point on $N$. By continuity, the same property holds at $m$. Denote 
	\[V:= \gamma'\mid_m.\]
	Since the extension of $\gamma$ is a geodesic, it is the image of the exponential map of $tV$ for parameter $t$ in an open interval of the real line containing $0$. 
	
	Furthermore, $J\nabla f$ is a Killing vector field generating an one-parameter group of isometries. As an isometry commutes with the exponential map, for any $U$ in a small neighborhood of the origin on the plane $\text{span}\{V, JV\}$, its image $p$ under the exponential map shares the same property that  
	\[0\neq \nabla f \parallel \nabla \SS.\]
	As a non-degenerate connected component of $\mathcal{S}_1$ is totally determined by a point and the flows of $\nabla f$ and $J\nabla f$, the image of the exponential map of a small punctured domain in the plane determined by $V, JV$ coincides with $N$. Extending this domain as much as possible gives the conclusion.  
	
\end{proof}

\begin{corollary}
	\label{atmost2}
	Let $m\in \mathcal{S}_0$ be non-degenerate. Then it connects to at most two branches of singularities in $\mathcal{S}_1$.
\end{corollary}
\begin{proof}
	The proof is by contradiction. Suppose $m$ is connected to three different branches of singularities in $\mathcal{S}_1$. By Lemma \ref{fundlem}, each is the image of the exponential map of a domain in a plane determined by $V_i, JV_i$ for $i=1, 2, 3$. 
	
	Since they are different branches, $V_i\notin \text{span}\{V_j, JV_j\}$ for $i\neq j$. Since each $V_i$ and $JV_i$ are eigenvectors of $\He f$ and $\He f$ has two eigenvalues, each of multiplicity two, there exists, without loss of generality,
	\[\He V_1= \mu V_1 \text{   and  } \He V_2=\mu V_2.\] 
	Since $\{V_1, JV_1, V_2, JV_2\}$ forms a basis of $T_pM$, $\He f$ is a multiple of the identity. By Lemma \ref{chardege0}, $m$ is degenerate, a contradiction.

\end{proof}

\begin{proof}[Proof of Theorem \ref{S1elliptic}]
	The proof is by contradiction. Let $p\in \mathcal{S}_1$. Supposed that the conclusion fails then, $p$ has an hyperbolic bloc. By Lemma \ref{branchhyper}, there is an embedded line segment of critical values in the in the interior of the image of $\Phi(U)$ for a small neighborhood $U$ of $m$. 
	
	Furthermore, the branch is detected by the property that, at each point, 
		\[0\neq \nabla f \parallel \nabla \SS.\]
	Thus, the corresponding two-dimensional submanifold is totally geodesic and has a warped product structure by Lemmas \ref{totalgeo} and \ref{curK}. In particular, the embedded line segment corresponds to the image, via $\Phi$, of an integral curve of $f$. Thus, we can extend this segment as long as $\nabla f\neq 0$ since the system has no degenerate singular points. 
	
	Since $f$ is proper and bounded below or above, by compactness, there must be a limit point $m$ which is in $\mathcal{S}_0$ by Lemma \ref{limitS0}. By Theorem \ref{S0}, $m$ is of type elliptic-elliptic. By Lemma \ref{branchhyper}, $m$ is connected to two branches of elliptic singularities in $\mathcal{S}_1$. In addition to the hyperbolic branch, $m$ is connected to three different branches of singularities, a contradiction to Lemma \ref{atmost2}.

\end{proof}

\subsection{Proof of Main Theorem}
We follow the strategy of \cite{VNS07, PV09, PRV15}. 

First, we describe the image of the moment map. 
\begin{lemma}
	\label{imagemomentmap}
	Let $(M, g, J, f, \lambda)$ be a KGRS in real dimension four. Assume $f$ is proper. We have the following:
	\begin{enumerate}
		\item The functions $\SS^{+}(x):= \max_{f^{-1}(x)} \SS$, $\SS^{-}(x):= \min_{f^{-1}(x)} \SS$ are continuous;
		\item The image $B=\Phi(M)$ is the domain defined by 
		\[B=\{(x, y)\in \mathbb{R}^2, f_{min}\leq x \leq f_{max} \text{ and } \SS^-(x)\leq  y\leq \SS^+(x)\}.\]
		Consequently, $B$ is simply connected. 
	\end{enumerate}
\end{lemma}
\begin{proof}
	By Morse theory, nearby regular level sets of $f$ are diffeomorphic to each other. Thus, a discontinuity of $\SS^{\pm}$ could only appear at a critical value of $f$. However, if $\nabla f =0$ on a connected component, then so is $\nabla \SS$ via the soliton equation $\Rc(\nabla f)=\frac{1}{2}\nabla \SS$. Consequently, $\SS$ is constant on such a component and there is no vertical segment in the image of the moment map. That proves the first statement. 
	
	Since $f$ is proper, Lemma \ref{connectedlevel} implies that each fiber $f^{-1}(x)$ is compact and connected. Thus, $\SS^{\pm}$(x) is finite. The result then follows from the previous one.   
\end{proof}

\begin{lemma}
	\label{connectefiberPhi}
	Let $(M, g, J, f, \lambda)$ be a KGRS in real dimension four. Supposed that the integrable Hamiltonian system $(M, \omega, f, \SS)$ is non-degenerate and $f$ is proper and bounded below or above. Then each fiber of the momentum map is connected.
\end{lemma}
\begin{proof}
	By Theorem \ref{S0}, all singularities of rank 0 are of type elliptic-elliptic. By Theorem \ref{S1elliptic}, all singularities of rank 1 are elliptic. Thus, the integrable system is almost-toric by the definition 3.4 of \cite{PRV15}. Furthermore, by Lemma \ref{connectedlevel}, each fiber of $f$ is connected. The result then follows from Theorem 4.7 of \cite{PRV15}. 
\end{proof}

\begin{proof}[Proof of Theorem \ref{main2}]
First, one recall that, for $\lambda\neq 0$, $\SS$ is bounded below \cite{chenbl09, zhang09completeness}. Thus, the soliton equation (\ref{nablafandS}) implies that $f$ is either bounded below ($\lambda>0$) or above ($\lambda<0$). For $\lambda=0$, O. Munteanu and J. Wang observed that the soliton structure must be connected at infinity \cite{MW11MMS}. Since $f$ is proper, it must be either bounded below or above. Thus, in all cases, our hypothesis implies $f$ is bounded above or below. In addition, Corollary \ref{fMorse} states that $f$ is Morse.   
	
Next, by Theorems \ref{S0} and \ref{S1elliptic}, $(M, \omega, f, \SS)$ is a non-degenerate integrable system with only elliptic singularities. Thus, it is an almost-toric system. By Lemma \ref{imagemomentmap}, the image of the moment map is simply connected. Then Lemma \ref{connectefiberPhi} and Proposition 2.9 of \cite{VNS07} imply that the set of regular values of the momentum map is simply-connected (there is no rank $0$ focus-focus singularity). Applying \cite[Prop 2.12]{VNS07} yields the conclusion. 
\end{proof}

\def\cprime{$'$}
\bibliographystyle{plain}
\bibliography{bioMorse}
\end{document}